\documentclass[11pt]{article}
\usepackage{amsmath}
\usepackage{amsthm}
\usepackage{amssymb}
\usepackage{amscd}
\usepackage{relsize}
\usepackage{enumerate}
\usepackage{eufrak}
\newtheorem{teo}{Theorem}
\newtheorem{prop}{Proposition}
\newtheorem{cor}[teo]{Corollary}
\newtheorem{lema}[teo]{Lemma}

\newcommand{\CC}{\mathbb{C}}
\newcommand{\RR}{\mathbb{R}}
\newcommand{\ZZ}{\mathbb{Z}}

\newcommand{\Ac}{\mathcal{A}}

\newcommand{\Hc}{\mathcal{H}}
\newcommand{\Tc}{\mathcal{T}}
\newcommand{\Lc}{\mathcal{L}}

\newcommand{\Uc}{\mathcal{U }}

\newcommand{\rank}{\operatorname{rank}}
\newcommand{\espan}{\operatorname{span}}

\title{{\bf Knit product of finite groups and sampling}}
\author{
{\bf Antonio~G. Garc\'{\i}a}\thanks{E-mail:\texttt{agarcia@math.uc3m.es}}, \, 
{\bf Miguel~A. Hern\'andez-Medina}\thanks{E-mail:\texttt{miguelangel.hernandez.medina@upm.es}}\, 
{\bf and  Alberto Ibort}\thanks{E-mail:\texttt{albertoi@math.uc3m.es}} \,\, }
\date{}
\begin{document}
\maketitle
\begin{itemize}
\item[*\ddag] Departamento de Matem\'aticas, Universidad Carlos III de Madrid,
 Avda. de la Universidad 30, 28911 Legan\'es-Madrid, Spain.
\item[\dag] Information Processing and Telecommunications Center, Universidad Polit\'ecnica de Madrid, Departamento de Matem\'atica Aplicada a las Tecnolog\'{\i}as de la Informaci\'on y las Comunicaciones, E.T.S.I.T., Avda. Complutense 30, 28040 Madrid, Spain.
\end{itemize}
%%%%%%%%%%%%%%%%%%%%
\begin{abstract}
A finite sampling theory associated with a unitary representation of a finite non Abelian group $\mathbf{G}$ on a Hilbert space is stablished. The non Abelian group 
$\mathbf{G}$ is a knit product $\mathbf{N}\bowtie \mathbf{H}$ of two finite subgroups $\mathbf{N}$ and $\mathbf{H}$. Sampling formulas where the samples are indexed by either $\mathbf{N}$ or $\mathbf{H}$ are obtained. Using suitable expressions for the involved samples, the problem is reduced to obtain dual frames in the Hilbert space $\ell^2(\mathbf{G})$ having a unitary invariance property; this is done by using matrix analysis techniques. An example involving dihedral groups illustrates the obtained sampling results.
\end{abstract}
%%%%%%%%%%%%%%%%%%%%%%%%%%%%%%%%%%%%%%%%%%%%%%%%%%%%%%%%%%%%%%%%%%%%%%
{\bf Keywords}: Knit product of groups; Unitary representation of a group; Finite unitary-invariant subspaces; Finite frames; Dual frames; Left-inverses; Sampling expansions.

\noindent{\bf AMS}: 20C40; 42C15; 94A20.
%%%%%%%%%%%%%%%%%%%%%

%%%%%%%%%%%%%%%%%%
\section{Statement of the problem}
\label{section1}
%%%%%%%%%%%%%%%%%%
In this paper an abstract sampling theory associated with a unitary representation of a non Abelian group $\mathbf{G}$, which is the knit-product  $\mathbf{G}=\mathbf{N}\bowtie \mathbf{H}$ of two finite subgroups $\mathbf{N}$ and $\mathbf{H}$, on a Hilbert space $\Hc$ is obtained. Specifically, given $\mathbf{G} \ni g \mapsto U(g)\in \Uc(\Hc)$ a unitary representation of  the group $\mathbf{G}$ on a Hilbert space $\Hc$, for a fixed vector $\mathsf{a}\in  \Hc$ the sampling subspace 
$\Ac_\mathsf{a}:=\espan \{U(g)\mathsf{a}\}_{g\in \mathbf{G}}$ of $\Hc$ is considered. The aim is to obtain sampling formulas in $\Ac_\mathsf{a}$  having a compatible $U$-structure and involving samples indexed by  either $\mathbf{N}$ or  $\mathbf{H}$. Namely, fixing $\kappa$ vectors $\mathsf{b}_k$ in $\Hc$, which do not necessarily belong to $\Ac_\mathsf{a}$, for each $\mathsf{f}\in \Ac_\mathsf{a}$ we consider the data samples
\begin{equation}
\label{defsamples}
\Lc_k \mathsf{f}(\nu):=\big\langle \mathsf{f}, U(\nu)\mathsf{b}_k \big\rangle_{\Hc}\,, \,\, \nu\in \mathbf{N},\quad \text{or}\quad \Lc_k \mathsf{f}(\tau):=\big\langle \mathsf{f}, U(\tau)\mathsf{b}_k \big\rangle_{\Hc}\,, \,\, \tau\in \mathbf{H}\,,
\end{equation}
for $k=1,2,\dots,\kappa$. Thus, the aim is to study the existence of sampling formulas in $\Ac_\mathsf{a}$ having the form
\begin{equation}
\label{fsampling}
\mathsf{f}=\sum_{k=1}^\kappa \sum_{\nu \in \mathbf{N}}\Lc_k\mathsf{f}(\nu)\,U(\nu)\mathsf{c}_k\quad \text{or} \quad \mathsf{f}=\sum_{k=1}^\kappa \sum_{\tau \in \mathbf{H}}\Lc_k\mathsf{f}(\tau)\,U(\tau)\mathsf{d}_k\quad \text{for each $\mathsf{f}\in \Ac_\mathsf{a}$}\,,
\end{equation}
for some vectors $\mathsf{c}_k$ or $\mathsf{d}_k$, $k=1,2, \dots,\kappa$, in $\Ac_\mathsf{a}$.

\medskip

All the classical sampling formulas have a structure compatible with an underlying group which defines the sampling space. Thus, the Paley-Wiener space $PW_\pi$, consisting of bandlimited functions to $[-\pi, \pi]$ in $L^2(\RR)$, is a particular shift-invariant subspace with underlying group $(\ZZ,+)$ represented unitarily on $L^2(\RR)$ by translations, and  the famous Shannon sampling formula reads
\[
f(t)=\sum_{n\in \ZZ} f(n)\,\dfrac{\sin \pi(t-n)}{\pi(t-n)}\,,\quad t\in \RR\,.
\]
Whether we sample at the subgroup $m\ZZ$ we must consider $\kappa \geq m$ sampling channels as proposed in formulas \eqref{fsampling}. In classical finite or infinite sampling the involved samples usually are averages, as in \eqref{defsamples}, or pointwise samples, and the underlying group is a locally compact Abelian group allowing the use of a classical Fourier analysis. See, among others, Refs.~\cite{hector:14,hector:16,frazier:94,garcia:00,garcia:06,garcia:15,kluvanek:65,stankovic:06}. 

For non Abelian groups a classical Fourier transform is not available and consequently other techniques should be considered as in \cite{barbieri:15}; for the finite case, see, for instance, Part II of Ref.~\cite{terras:99}. Non Abelian groups profusely appear in the mathematical literature, having important applications in different fields as geometry, signal processing, mathematical physics, chemistry, etc.
\cite{dodson:07,kettle:07,sinha:10,stankovic:05,stankovic:06,terras:99}. An important number of these groups have in common that they are obtained from Abelian ones. This is the case of semidirect and knit products of groups, which include as examples the dihedral groups $D_{2N}$, the infinite dihedral group $D_\infty$, crystallographic groups, Euclidean or Heisenberg motion groups, etc.
 
In this paper we propose a finite sampling theory associated with  a non Abelian group obtained from knit product and samples like in \eqref{defsamples}; the obtained results are just based on linear algebra techniques: finite frames and left-inverse matrices.   As far as we know this is a novel approach to finite sampling related to a non Abelian group that intends to be the first step in this direction.

\medskip

Next we briefly detail the mathematical techniques used throughout the paper; they rely on the expression of the samples as $\Lc_k \mathsf{f}(\nu)=\big\langle \boldsymbol{\alpha}, \mathbf{g}_{k,\nu} \big\rangle_{\ell^2(\mathbf{G})}$, where $\boldsymbol{\alpha}=(\alpha_g)_{g\in \mathbf{G}}$ is the coefficients vector of 
$\mathsf{f} \in \Ac_\mathsf{a}$ in the basis $\{U(g)\mathsf{a}\}_{g\in \mathbf{G}}$, and $\mathbf{g}_{k,\nu}\in \ell^2(\mathbf{G})$ is  a vector obtained from the cross-covariance of the finite sequences 
$\big\{U(g)\mathsf{a}\big\}_{g\in G}$ and $\big\{U(h)\mathsf{b}_k\big\}_{h\in \mathbf{G}}$; the cross-covariance used here generalizes the concept introduced by Kolmogorov in \cite{kolmogorov:41}. Thus, the stable recovery of $\boldsymbol{\alpha} \in \ell^2(\mathbf{G})$, or equivalently of  $\mathsf{f}\in \Ac_\mathsf{a}$, from the given data $\{\Lc_k \mathsf{f}(\nu)\}_{k,\nu}$, i.e., the existence of two constants $0<A\le B$ such that
\[
A\|\mathsf{f}\|^2 \le \sum_{k,\nu} |\Lc_k \mathsf{f}(\nu)|^2 \le B\|\mathsf{f}\|^2\quad \text{for all $\mathsf{f}\in \Ac_\mathsf{a}$}\,,
\]
depends on whether $\{\mathbf{g}_{k,\nu}\}_{k,\nu}$ forms a frame for $\ell^2(\mathbf{G})$; or equivalently, if it forms a spanning set for $\ell^2(\mathbf{G})$ since $\Ac_\mathsf{a}$ is a finite dimensional subspace. Recall that a sequence $\{x_n\}$ is a frame for a separable Hilbert space $\Hc$ if there exist two constants $0<A\le B$, the frame bounds, such that
\[
A\|x\|^2 \leq \sum_n |\langle x, x_n \rangle|^2 \leq B \|x\|^2 \,\, \text{ for all } x\in \Hc \,.
\]
Given a frame $\{x_n\}$ for $\Hc$ the representation property of any vector $x\in \Hc$ as a series $x=\sum_n c_n x_n$ is retained, but, unlike the case of Riesz bases, the uniqueness of this representation (for overcomplete frames) is sacrificed. Suitable frame coefficients $c_n$ which depend continuously and linearly on $x$ 
are obtained by using the dual frames $\{y_n\}$ of $\{x_n\}$, i.e., 
$\{y_n\}$ is another frame for $\Hc$ such that $x=\sum_n \langle x, y_n \rangle \,x_n=\sum_n \langle x, x_n \rangle \,y_n $ for each $x\in \Hc$. For
more details on the frame theory see, for instance, the monograph \cite{ole:16} and references therein; see also Ref.~\cite{casazza:14} for finite frames.

In order to derive the left sampling formula in \eqref{fsampling} we assume that the subgroup 
$\mathbf{N}$  is Abelian (respectively $\mathbf{H}$ Abelian for deriving the right sampling formula); thus we construct dual frames of $\{\mathbf{g}_{k,\nu}\}_{k,\nu}$ having the special needed structure. This is achieved by constructing some specific left-inverses of the cross-covariance matrix $\mathbb{R}_{\mathsf{a},\mathbf{b}}$ that gathers all the cross-covariances information; thus, group theory meets matrix analysis. 

\medskip

The paper is organized as follows: in Section \ref{section2} we will succinctly review the knit product of groups; in Section \ref{section3}, after establishing the mathematical setting used throughout the paper we derive the main sampling results (Theorems \ref{teoGeneral} and \ref{teoGeneralH}). They are based on the existence of $G$-compatible left-inverses of the cross-covariance matrix $\mathbb{R}_{\mathsf{a},\mathbf{b}}$ and a construction for these left-inverses is also provided; finally, in Section \ref{section4} we put to work our results for a specific example involving the dihedral group $D_{2N}$ of symmetries of a regular $N$-sided polygon.

%%%%%%%%%%%%%%%%%%%%%%%%%%%%%%%%%%%%%%%%%%%%%%
\section{A brief on the knit product $\mathbf{G}=\mathbf{N} \bowtie \mathbf{H}$ of groups}
\label{section2}
%%%%%%%%%%%%%%%%%%%%%%%%%%%%%%%%%%%%%%%%%%%%%%
Let $\mathbf{G}$ be a group with identity $1_{\mathbf{G}}$, and let $\mathbf{N}$ and $\mathbf{H}$ be subgroups of $\mathbf{G}$ such that $\mathbf{G}=\mathbf{N}\mathbf{H}$ and $\mathbf{N}\cap \mathbf{H}=\{1_{\mathbf{G}}\}$, or, equivalently, for each $g\in \mathbf{G}$ there exists  a unique $h \in \mathbf{H}$ and a unique $n \in \mathbf{N}$ such that $g = nh $. In this case, $\mathbf{G}$ is said to be the {\em internal  knit product} (or {\em Zappa-Sz\'ep product}) of $\mathbf{N}$ and 
$\mathbf{H}$, and it is denoted by 
$\mathbf{G}=\mathbf{N}\bowtie \mathbf{H}$. For each $h\in \mathbf{H}$ and $n\in \mathbf{N}$ there exist $\alpha(n,h) \in \mathbf{N}$ and $\beta(n,h) \in \mathbf{H}$ such that $hn = \alpha(n,h)\beta(n,h)$. This defines mappings 
$\alpha^*  : \mathbf{H} \to Aut(\mathbf{N})$ and $\beta^* : \mathbf{N} \to Aut(\mathbf{H})$, where 
$\alpha^*(h)(n):=\alpha(n,h)$ and $\beta^*(n)(h):=\beta(n,h)$ for $n\in \mathbf{N}$ and $h\in \mathbf{H}$, and satisfying (see Refs.~\cite{brin:05,szep:50,zappa:42}).

\begin{enumerate}
  \item $\alpha^*$ is a group isomorfism and $\beta^*$ is a group anti-isomorfism (i.e. 
     $\beta_{n_1n_2}=\beta_{n_2}\beta_{n_1}$).
  \item $\alpha_h(n_1n_2)=\alpha_h(n_1)\alpha_{\beta_{n_1}(h)}(n_2)$ for each $h\in \mathbf{H}$ and  $n_1,n_2\in \mathbf{N}$,
  \item $\beta_n(h_1h_2)=\beta_{\alpha_{h_2}(n)}(h_1)\beta_n(h_2)$ for each $n\in \mathbf{N}$ and $h_1,h_2\in \mathbf{H}$\,,
\end{enumerate}
where we have denoted $\alpha_h(n):=\alpha^*(h)(n)$ and $\beta_n(h):=\beta^*(n)(h)$. 

\medskip

From now on, we denote the order of the involved groups as $\mathfrak{g}:=|\mathbf{G}|$, $\mathfrak{n}:=|\mathbf{N}|$ and $\mathfrak{h}:=|\mathbf{H}|$ respectively.
In case $\mathbf{N}$ is a normal subgroup of $\mathbf{G}$ then $\alpha_h(n):=hnh^{-1}$ and $\beta_n:= \mathrm{id}_H$, the corresponding knit product $\mathbf{N} \bowtie \mathbf{H}$ coincides with the {\em internal semidirect product} of $\mathbf{N}$ and $\mathbf{H}$, denoted by $\mathbf{G}=\mathbf{N}\rtimes \mathbf{H}$.

Whether $\mathbf{G}=\mathbf{N}\bowtie \mathbf{H}$ we can choose exactly one element of $\mathbf{N}$ in each left coset of the quotient set 
$\mathbf{G}/\mathbf{H}$. Thus, in case $\mathbf{N}=\{\nu_1=1_{\mathbf{G}},\, \nu_2,\dots,\nu_\mathfrak{n}\}$ we can describe the quotient 
set $\mathbf{G}/\mathbf{H}$ as
\begin{equation}
\label{G/H}
\mathbf{G}/\mathbf{H}=\big\{[1_{\mathbf{G}}=\nu_1], [\nu_2], \cdots, [\nu_\mathfrak{n}] \big\}\,.
\end{equation}
There is an external version of the knit product for groups. In this case, we have two groups $\mathbf{N}$ and $\mathbf{H}$ which are not known to be subgroups of a given group and mappings $\alpha: \mathbf{H}\times \mathbf{N} \rightarrow \mathbf{N}$ and 
$\beta: \mathbf{H}\times \mathbf{N} \rightarrow \mathbf{H}$ satisfying the  properties 1--3 above. On the product $\mathbf{N}\times \mathbf{H}$ we define a product law
\[
(n_1, h_1) (n_2, h_2) = \big(n_1 \alpha_{h_1}(n_2), \beta_{n_2}(h_1) h_2\big)\quad \text{ for any $n_1,\, n_2\in \mathbf{N}$ and $h_1,\, h_2\in \mathbf{H}$}\,.
\]
Thus, $(1_\mathbf{N},1_\mathbf{H})$ is the identity, and $(n,h)^{-1}=(\alpha_{h^{-1}}(n^{-1}),\beta_{n^{-1}}(h^{-1})$ for any $(n,h)\in \mathbf{N}\times \mathbf{H}$.
With this product the set $\mathbf{N}\times \mathbf{H}$ is a group called {\em external knit product} of the groups $\mathbf{N}$ and $\mathbf{H}$, and it is denoted again as $\mathbf{N}\bowtie \mathbf{H}$. The subsets $\mathbf{N}\times \{1_{\mathbf{H}}\}$ and $\{1_{\mathbf{N}}\}\times \mathbf{H}$ are subgroups of $\mathbf{N}\bowtie \mathbf{H}$ isomorphic to $\mathbf{N}$ and $\mathbf{H}$ respectively. Clearly, the external knit product of the groups $\mathbf{N}$ and $\mathbf{H}$ coincides with the internal knit product of subgroups  $\mathbf{N}\times \{1_{\mathbf{H}}\}$ and $\{1_{\mathbf{N}}\}\times \mathbf{H}$.

%%%%%%%%%%%%%%%%%%%%%%%%%%%%%%%%%%%%%%%%%%
\section{Sampling associated with a unitary representation of the group $\mathbf{G}=\mathbf{N} \bowtie \mathbf{H}$}
\label{section3}
%%%%%%%%%%%%%%%%%%%%%%%%%%%%%%%%%%%%%%%%%%
As it was said in the introduction we begin this section by establishing the mathematical setting used throughout the paper.
%%%%%%%%%%%%%%%%%%%%%%%%%%%%%%%%%%%%%%%%%%
\subsection{The mathematical setting}
%%%%%%%%%%%%%%%%%%%%%%%%%%%%%%%%%%%
Let $\mathbf{G}$ be a finite, not necessarily Abelian, group of order $\mathfrak{g}$ with identity element $1_{\mathbf{G}}$. Let $\mathbf{G} \ni g \mapsto U(g)\in \Uc(\Hc)$ be a {\em unitary representation} of  $\mathbf{G}$ on a Hilbert space $\Hc$, i.e., a homomorphism from the group $\mathbf{G}$ into the group $\Uc(\Hc)$ of unitary operators on $\Hc$, i.e., a map satisfying  $U(gg')=U(g)U(g')$ and $U(1_{\mathbf{G}})=I_{\Hc}$.  

\medskip

From now on, for a fixed vector $\mathsf{a}\in  \Hc$ we consider the subspace $\Ac_\mathsf{a}$  of $\Hc$ spanned by the vectors $U(g)\mathsf{a}$, $g\in \mathbf{G}$, i.e., $\Ac_\mathsf{a}:=\espan \{U(g)\mathsf{a}\}_{g\in \mathbf{G}}$. In case this set is linearly independent in $\Hc$, each $\mathsf{f}\in \Ac_\mathsf{a}$ can be expressed uniquely as the  expansion $\mathsf{f}=\sum_{g\in \mathbf{G}} \alpha_g \,U(g)\mathsf{a}$, with $\alpha_g\in \CC$.

There is a close relationship between the finite sequence $\{U(g)\mathsf{a}\}_{g\in \mathbf{G}}$ in the Hilbert space $\Hc$ and the so-called {\em stationary sequences} (see Kolmogorov \cite{kolmogorov:41}). We say that the finite sequence $\{\mathsf{a}_g\}_{g\in \mathbf{G}}$ in $\Hc$ is (left) {\em $\mathbf{G}$-stationary} if 
\[
\langle \mathsf{a}_g, \mathsf{a}_{g'} \rangle_\Hc=\langle \mathsf{a}_{hg}, \mathsf{a}_{hg'} \rangle_\Hc\,,\quad \text{for all $g, g', h\in \mathbf{G}$}\,.
\]
In other words, the inner product $\langle \mathsf{a}_g, \mathsf{a}_h\rangle_\Hc$ only depends on $h^{-1}g$. Then, it is easy to deduce that  there exists a unitary representation $g\mapsto U(g)$ of the group $\mathbf{G}$ on $\Hc$ and $\mathsf{a}\in  \Hc$ such that $\mathsf{a}_g  =U(g)\mathsf{a}$, $g\in \mathbf{G}$.
We define the {\em auto-covariance} of the finite sequence $\{U(g)\mathsf{a}\}_{g\in \mathbf{G}}$ as the positive semidefinite function:
\[
r_\mathsf{a}(g):=\big\langle U(g)\mathsf{a}, \mathsf{a}\big\rangle_\Hc\,, \quad g\in \mathbf{G}\,.
\] 
Similarly, we define the {\em cross-covariance} between the finite sequences $\{U(g)\mathsf{a}\}_{g\in \mathbf{G}}$ and $\{U(g)\mathsf{b}\}_{g\in \mathbf{G}}$ where $\mathsf{a}, \mathsf{b} \in\Hc$ as
\[
r_{\mathsf{a}, \mathsf{b}}(g):=\big\langle U(g)\mathsf{a}, \mathsf{b}\big\rangle_\Hc\,, \quad g\in \mathbf{G}\,.
\]
Note that $r_{\mathsf{a}, \mathsf{b}}(g)=\overline{r_{\mathsf{b}, \mathsf{a}}(g^{-1})}$ for $\mathsf{a}, \mathsf{b} \in\Hc$ and $g\in \mathbf{G}$.

\begin{prop}\label{Rinv}
Let $\mathbb{R}_\mathsf{a}$ denote the  square matrix of order $\mathfrak{g}$ defined from the auto-covariance $r_\mathsf{a}$ as
$\mathbb{R}_\mathsf{a}:=\big(r_\mathsf{a}(t^{-1}g) \big)_{t,g\in \mathbf{G}}$. Then, the set of vectors $\{U(g)\mathsf{a}\}_{g\in \mathbf{G}}$ is linearly independent in $\Hc$ if and only if $\det \mathbb{R}_\mathsf{a} \neq 0$.
\end{prop}
\begin{proof}
If $\det \mathbb{R}_\mathsf{a} = 0$ then there exists a vector $\boldsymbol{\lambda}=(\lambda_g)_{g\in \mathbf{G}} \in \CC^\mathfrak{g}$ such that  $\boldsymbol{\lambda} \neq \mathbf{0}$ and  $\mathbb{R}_\mathsf{a}  \boldsymbol{\lambda} = 0$. Thus  $\sum_{g\in \mathbf{G}} \lambda_g\,U(g)\mathsf{a}$ is orthogonal 
to $U(g)\mathsf{a}$ for all $g\in \mathbf{G}$ so that  $\sum_{g} \lambda_g\,U(g)\mathsf{a}=0$.  Conversely, if $\sum_{g\in G} \lambda_g\,U(g)\mathsf{a}=0$ for some $\boldsymbol{\lambda} \neq \mathbf{0}$ then the inner product  in the above expression  with each $U(t)\mathsf{a}$, $t\in \mathbf{G}$, yields   
$\mathbb{R}_\mathsf{a}  \boldsymbol{\lambda} = 0$.
\end{proof}
%%%%%%%%%%%%%%%%%%%%%%%%%%%
\subsection*{The isomorphism $\Tc^\mathbf{G}_\mathsf{a}$}
%%%%%%%%%%%%%%%%%%%%%%%%%%%
Consider the group algebra $\mathbb{C}[\mathbf{G}]$, that is, the complex linear space generated by the elements of the group $\mathbf{G}$.  Thus $\mathbb{C}[\mathbf{G}]$ has dimension $\mathfrak{g}$ and its elements can be identified with the space of functions $ \boldsymbol{\alpha} \colon \mathbf{G} \to \mathbb{C}$, $g \mapsto \boldsymbol{\alpha}(g)$; in brief $\boldsymbol{\alpha}=\big( \boldsymbol{\alpha}(g)\big)_{g\in \mathbf{G}}$.   In the finite situation we are dealing with, such functions are all obviously integrable and square integrable, hence it can be identified with $\ell^2(\mathbf{G})$ endowed with its natural inner product 
$\langle \boldsymbol{\alpha}, \boldsymbol{\beta} \rangle$ which becomes a Hilbert space isomorphic to $\CC^\mathfrak{g}$. 

The Hilbert space $\ell^2(\mathbf{G})$ supports a natural unitary representation of $\mathbf{G}$ called the {\em left regular representation} $L_s$, $s\in \mathbf{G}$, defined by
\[
L_s \boldsymbol{\alpha}(g)=\boldsymbol{\alpha}(s^{-1}g)\quad \mathrm{for} \, \, s,g\in \mathbf{G}\,.
\]
Next we define the natural surjective linear map $\Tc^\mathbf{G}_\mathsf{a}$ between $\ell^2(\mathbf{G})$ and $\Ac_\mathsf{a}$:
\begin{equation}
\label{iso}
\begin{array}[c]{ccll}
\Tc^\mathbf{G}_\mathsf{a}: & \ell^2(\mathbf{G}) & \longrightarrow & \mathcal{A}_\mathsf{a}\\
       & \boldsymbol{\alpha} & \longmapsto & \mathsf{f}=\displaystyle{\sum_{g\in \mathbf{G}} \boldsymbol{\alpha}(g)\,U(g)\mathsf{a}}\,.
\end{array}
\end{equation}
If matrix $\mathbb{R}_\mathsf{a}$ is invertible, because of Prop. \ref{Rinv},  the above map $\Tc^\mathbf{G}_\mathsf{a}$  becomes an isomorphism and it has the following {\em shifting property} with respect to the left representation $L_s$:

\begin{prop}
For any $s\in \mathbf{G}$ and $\boldsymbol{\alpha}\in \ell^2(\mathbf{G})$ we have that
\begin{equation}
\label{shiftgroup}
\Tc^\mathbf{G}_\mathsf{a}\big(L_s \boldsymbol{\alpha}\big)=U(s)\,\Tc^\mathbf{G}_\mathsf{a} (\boldsymbol{\alpha})
\end{equation}
\end{prop}
\begin{proof}
Indeed, denoting $g'=s^{-1}g$ we have
\[
\begin{split}
\Tc^\mathbf{G}_\mathsf{a}\big(L_{s} \boldsymbol{\alpha}\big)&= \sum_{g\in \mathbf{G}} \boldsymbol{\alpha}(s^{-1}g) U(g)\mathsf{a}=\sum_{g'\in \mathbf{G}} \boldsymbol{\alpha}(g') U(sg')\mathsf{a} \\
&=\sum_{g'\in \mathbf{G}} \boldsymbol{\alpha}(g') U(s)U(g')\mathsf{a}=U(s)\,\Tc^\mathbf{G}_\mathsf{a} (\boldsymbol{\alpha})
\end{split}
\]
\end{proof}
Next two sections are devoted to obtain the sampling results:
%%%%%%%%%%%%%%%%%%%%%%%%%%%%%%%%%%%%%%%%
\subsection{Case of $\mathbf{N}$ Abelian subgroup: sampling indexed by $\mathbf{N}$}
%%%%%%%%%%%%%%%%%%%%%%%%%%%%%%%%%%%%%%%%
Having in mind the description of $\mathbf{G}/\mathbf{H}$ in \eqref{G/H}, we write the group $\mathbf{G}$  as $\mathbf{G}=\bigcup_{n=1}^\mathfrak{n}\nu_n^{-1}\mathbf{H}$. In the sequel we fix the way of writing the elements of the group $\mathbf{H}$; this will be important for maintaining the structure of the matrices $\mathbb{R}_{\mathsf{a},\mathsf{b}_k}$ introduced below.  We also need the Abelian character of subgroup $\mathbf{N}$  to get a block symmetry for the matrices $\mathbb{R}_{\mathsf{a},\mathsf{b}_k}$.

\medskip

Fixed $\kappa$ elements $\mathsf{b}_k \in\Hc$, $k=1,2,\dots,\kappa$, not necessarily in $\Ac_\mathsf{a}$, for each $\mathsf{f}\in \Ac_\mathsf{a}$
we define its {\em generalized samples}, indexed by the elements in $\mathbf{N}$, as
\begin{equation}
\label{gsamples}
\Lc_k \mathsf{f} (\nu_n):=\big\langle \mathsf{f},U(\nu_n)\mathsf{b}_k\big \rangle_\Hc\,, \quad \text{$n=1, 2, \dots, \mathfrak{n}$ and $k=1, 2,  \dots,\kappa$}\,.
\end{equation}
 Notice that the expression for the generalized samples (\ref{gsamples}) can be seen as an straightforward generalization of the convolution involving the sampled vector $\mathsf{f}\in \Ac_\mathsf{a}$ and the vectors $\mathsf{b}_k \in \Hc$.

Because in general to recover any $\mathsf{f} \in \Ac_\mathsf{a}$ we need at least $\mathfrak{g}$ samples, if the samples are indexed by elements in $\mathbf{N}$, we will need at least $\mathfrak{n}\kappa\geq \mathfrak{g}=\mathfrak{n}\mathfrak{h}$ samples, i.e., $\kappa\geq \mathfrak{h}$.

\medskip

The main aim of this paper is to recover any $\mathsf{f}\in\Ac_\mathsf{a}$ by means of its generalized samples in \eqref{gsamples} by means of  a sampling formula which takes care of the unitary structure of $\Ac_\mathsf{a}$. 

To this end, we first obtain an alternative expression for $\Lc_k\mathsf{f}(\nu_n)$ with 
$n=1, 2, \dots, \mathfrak{n}$ and $k=1, 2, \dots, \kappa$\,. Namely, introducing the expression of $\mathsf{f}\in \Ac_\mathsf{a}$ in \eqref{gsamples} we get
\begin{equation}
\label{genSamples}
\begin{split}
\Lc_k\mathsf{f}(\nu_n)&=\Big\langle \sum_{g\in \mathbf{G}} \alpha_g \,U(g)\mathsf{a}, U(\nu_n)\mathsf{b}_k \Big\rangle_\Hc 
         =\sum_{g\in \mathbf{G}} \alpha_g \,\big\langle U (g)\mathsf{a},U(\nu_n)\mathsf{b}_k\big\rangle_\Hc \\ 
         &= \big\langle(\alpha_g)_{g\in \mathbf{G}},(\overline{\langle U(g)\mathsf{a},U(\nu_n)\mathsf{b}_k\rangle})_{g\in \mathbf{G}}\big\rangle_{\ell^2(\mathbf{G})} =\big\langle \boldsymbol{\alpha},\mathbf{g}_{k,\nu_n}\big\rangle_{\ell^2(\mathbf{G})}\,,
\end{split}
\end{equation}
where $\boldsymbol{\alpha}=(\alpha_g)_{g\in \mathbf{G}}$ and $\mathbf{g}_{k,\nu_n}=(\overline{\langle U(g)\mathsf{a},U(\nu_n)\mathsf{b}_k\rangle })_{g\in \mathbf{G}}$ belong to $\ell^2(\mathbf{G})$. 
The vectors $\mathbf{g}_{k,\nu_n}\in \ell^2(\mathbf{G})$, $k=1,2,\dots,\kappa$ and $n=1, 2,\dots, \mathfrak{n}$, can be expressed in terms of the cross-covariances 
$r_{\mathsf{a},\mathsf{b}_k}$ as 
\[
  \mathbf{g}_{k,\nu_n}=\Big(\,\overline{\langle U(g)\mathsf{a},U(\nu_n)\mathsf{b}_k \rangle}\,\Big)_{g\in \mathbf{G}}=
         \Big(\,\overline{\langle U(\nu_n^{-1}g)\mathsf{a},\mathsf{b}_k \rangle}\,\Big)_{g\in \mathbf{G}} =
         \Big(\,\overline{r_{\mathsf{a},\mathsf{b}_k}}(\nu_n^{-1}g)\,\Big)_{g\in \mathbf{G}} \,.
\]
Having in mind the expression \eqref{genSamples} for the samples and the isomorphism $\Tc_\mathsf{a}^\mathbf{G}$ defined in \eqref{iso} we deduce the following result (see also the finite frame theory \cite{casazza:14}):
\begin{prop}
\label{prop3}
Any $\mathsf{f}\in \Ac_\mathsf{a}$ can be recovered from its samples $\big\{\Lc_k \mathsf{f} (\nu_n) \big\}_{\substack{k=1,2,\ldots, \kappa \\ n=1,2,\ldots,\mathfrak{n}}}$ if and only if the set of vectors $\big\{\mathbf{g}_{k,\nu_n}\big\}_{\substack{k=1,2,\ldots, \kappa \\ n=1,2,\ldots,\mathfrak{n}}}$ in $\ell^2(\mathbf{G})$ form a spanning set (a frame) for $\ell^2(\mathbf{G})$. 
\end{prop}
Equivalently, the $\mathfrak{g}\times \mathfrak{n} \kappa$ matrix having columns $\mathbf{g}_{k,\nu_n}$, $k=1,2,\dots,\kappa$ and $n=1, 2,\dots, \mathfrak{n}$, i.e.,
\begin{equation}
\label{matrixG}
\Big( \mathbf{g}_{1,\nu_1} \dots \mathbf{g}_{1,\nu_\mathfrak{n}}\,\,  \mathbf{g}_{2,\nu_1}  \dots \mathbf{g}_{2,\nu_\mathfrak{n}} 
       \dots  \mathbf{g}_{\kappa,\nu_1}  \dots \mathbf{g}_{\kappa,\nu_\mathfrak{n}}\Big)
\end{equation} 
has rank $\mathfrak{g}$. Hence, we deduce that $\mathfrak{g}\leq \mathfrak{n}\kappa$, that is, $\kappa \geq \mathfrak{h}$. The vector $\mathbf{g}_{k,\nu_n}$ can be written as 
\[
\mathbf{g}_{k,\nu_n}=\Big(\overline{r_{\mathsf{a},\mathsf{b}_k}}(\nu^{-1}_{n}\nu_1^{-1}\mathbf{H}),
    \overline{r_{\mathsf{a},\mathsf{b}_k}}(\nu^{-1}_{n}\nu_2^{-1}\mathbf{H}),\dots, \overline{r_{\mathsf{a},\mathsf{b}_k}}(\nu^{-1}_{n}\nu_\mathfrak{n}^{-1}\mathbf{H})\Big)^\top
\]
where $\nu_n\in \mathbf{N}$ and $\overline{r_{\mathsf{a},\mathsf{b}_k}}(\nu^{-1}_{n}\nu_r^{-1}\mathbf{H})$ is the row vector
\[
  \overline{r_{\mathsf{a},\mathsf{b}_k}}(\nu^{-1}_{n}\nu^{-1}_r \mathbf{H})=\Big(\overline{r_{\mathsf{a},\mathsf{b}_k}}(\nu^{-1}_{n}\nu_r^{-1}\tau_1),
                        \overline{r_{\mathsf{a},\mathsf{b}_k}}(\nu^{-1}_{n}\nu_r^{-1}\tau_2),\dots,\overline{r_{\mathsf{a},\mathsf{b}_k}}(\nu^{-1}_{n}\nu_r^{-1}\tau_{\mathfrak{h}})\Big)\,,
\]
being $\mathbf{H}=\{\tau_1=1_{\mathbf{G}},\, \tau_2,\dots,\tau_\mathfrak{h}\}$. For each $k=1,2,\dots,\kappa$, let $\mathbb{R}_{\mathsf{a},\mathsf{b}_k}$ be the $\mathfrak{n}\times \mathfrak{g}$ matrix 
\[
  \mathbb{R}_{\mathsf{a},\mathsf{b}_k}=
      \begin{pmatrix}
          r_{\mathsf{a},\mathsf{b}_k}(\nu^{-1}_{1}\nu_1^{-1}\mathbf{H})&
              r_{\mathsf{a},\mathsf{b}_k}(\nu^{-1}_{1}\nu_2^{-1}\mathbf{H})&\dots& 
                                            r_{\mathsf{a},\mathsf{b}_k}(\nu^{-1}_{1}\nu_\mathfrak{n}^{-1}\mathbf{H})\\
        r_{\mathsf{a},\mathsf{b}_k}(\nu^{-1}_{2}\nu_1^{-1}\mathbf{H})&
              r_{\mathsf{a},b_k}(\nu^{-1}_{2}\nu_2^{-1}\mathbf{H})&\dots& 
                                             r_{\mathsf{a},\mathsf{b}_k}(\nu^{-1}_2\nu_\mathfrak{n}^{-1}\mathbf{H})\\
          \vdots & \vdots & \cdots& \vdots\\
         r_{\mathsf{a},\mathsf{b}_k}(\nu^{-1}_\mathfrak{n}\nu_1^{-1}\mathbf{H})&
              r_{\mathsf{a},\mathsf{b}_k}(\nu^{-1}_\mathfrak{n}\nu_2^{-1}\mathbf{H})&\dots& 
                                              r_{\mathsf{a},\mathsf{b}_k}(\nu^{-1}_\mathfrak{n}\nu_\mathfrak{n}^{-1} \mathbf{H}) 
         \end{pmatrix}
\]
Since $\mathbf{N}$ is an Abelian subgroup of $\mathbf{G}$, the left cosets $\nu^{-1}_n\nu_r^{-1}\mathbf{H}$ and 
$\nu^{-1}_r\nu_n^{-1}\mathbf{H}$ coincide. As a consequence, $\mathbb{R}_{\mathsf{a},\mathsf{b}_k}$ is the block symmetric matrix
\begin{equation}
 \label{estructuraRCasoGeneral}
  \mathbb{R}_{\mathsf{a},\mathsf{b}_k}=
      \begin{pmatrix}
          r_{\mathsf{a},\mathsf{b}_k}(\nu^{-1}_1\nu_1^{-1}\mathbf{H})&
              r_{\mathsf{a},\mathsf{b}_k}(\nu_2^{-1}\nu^{-1}_1 \mathbf{H})&\dots& 
                                 r_{\mathsf{a},\mathsf{b}_k}(\nu_\mathfrak{n}^{-1}\nu^{-1}_1 \mathbf{H})\\
        r_{\mathsf{a},\mathsf{b}_k}(\nu_1^{-1}\nu^{-1}_2 \mathbf{H})&
              r_{\mathsf{a},\mathsf{b}_k}(\nu^{-1}_2\nu_2^{-1}\mathbf{H})&\dots& 
                                   r_{\mathsf{a},\mathsf{b}_k}(\nu_\mathfrak{n}^{-1}\nu^{-1}_2 \mathbf{H})\\
          \vdots & \vdots & \cdots& \vdots\\
         r_{\mathsf{a},\mathsf{b}_k}(\nu_1^{-1}\nu^{-1}_\mathfrak{n} \mathbf{H})&
     r_{\mathsf{a},\mathsf{b}_k}(\nu_2^{-1}\nu^{-1}_\mathfrak{n} \mathbf{H})&\dots& r_{\mathsf{a},\mathsf{b}_k}(\nu^{-1}_\mathfrak{n}\nu_\mathfrak{n}^{-1}\mathbf{H}) 
      \end{pmatrix}
\end{equation}
The matrix given in \eqref{matrixG} can be written as $\big(\mathbb{R}_{\mathsf{a},\mathsf{b}_1}^*\,\mathbb{R}_{\mathsf{a},\mathsf{b}_2}^*\,\dots \,\mathbb{R}_{\mathsf{a},\mathsf{b}_\kappa}^*\big)$, where, as usual, the symbol $*$ denotes the traspose conjugate matrix. Thus, Proposition \ref{prop3} can be restated in terms of the 
$\mathfrak{n}\kappa\times \mathfrak{g}$ {\em cross-covariance matrix} $\mathbb{R}_{\mathsf{a},\mathbf{b}}$  defined by
\begin{equation}
\label{matrixcc}
\mathbb{R}_{\mathsf{a},\mathbf{b}}:= \begin{pmatrix}       
\mathbb{R}_{\mathsf{a},\mathsf{b}_1} \\
\mathbb{R}_{\mathsf{a},\mathsf{b}_2}\\  
\vdots \\
\mathbb{R}_{\mathsf{a},\mathsf{b}_\kappa}   \\
\end{pmatrix}
\end{equation}
\begin{cor}
\label{cor1}
Any $\mathsf{f}\in \Ac_\mathsf{a}$ can be recovered from its samples $\big\{\Lc_k \mathsf{f} (\nu_n) \big\}_{\substack{k=1,2,\ldots, \kappa \\ n=1,2,\ldots,\mathfrak{n}}}$ if and only $\rank\, \mathbb{R}_{\mathsf{a},\mathbf{b}}=\mathfrak{g}$. 
\end{cor}
Besides, Equation \eqref{genSamples} can be expressed, for any $\mathsf{f}=\sum_{g\in \mathbf{G}}\alpha_g \,U(g)\mathsf{a}$ in $\Ac_\mathsf{a}$,  as
\[
   \begin{pmatrix}
     \Lc_k \mathsf{f}(\nu_1)&
     \Lc_k \mathsf{f}(\nu_2)&
          \cdots &
     \Lc_k \mathsf{f}(\nu_\mathfrak{n})
   \end{pmatrix}^\top=\mathbb{R}_{\mathsf{a},\mathsf{b}_k}\,\boldsymbol{\alpha}\,,
\]
where $\boldsymbol{\alpha}=(\alpha_g)_{g\in \mathbf{G}}$. As a consequence, we deduce the expression:
\begin{prop} 
For any $\mathsf{f}=\sum_{g\in \mathbf{G}}\alpha_g \, U(g)\mathsf{a}$  in $\Ac_\mathsf{a}$, consider its samples vector 
\begin{equation}
\Lc_{\operatorname{samp}}\mathsf{f}:=\big(\Lc_1 \mathsf{f}(\nu_1)\dots\Lc_1 \mathsf{f}(\nu_p)\cdots \Lc_\kappa \mathsf{f}(\nu_1)\dots \Lc_\kappa \mathsf{f}(\nu_\mathfrak{n})\big)^\top\,.
\end{equation}
Then, the matrix relationship
\begin{equation}
\label{samplematrix}
\Lc_{\operatorname{samp}}\mathsf{f}=\mathbb{R}_{\mathsf{a},\mathbf{b}}\,\boldsymbol{\alpha}
\end{equation}
holds, where $\boldsymbol{\alpha}=(\alpha_g)_{g\in \mathbf{G}}$ and $\mathbb{R}_{\mathsf{a},\mathbf{b}}$ is the $\mathfrak{n}\kappa\times \mathfrak{g}$ 
cross-covariances matrix defined in \eqref{matrixcc}.
\end{prop}

Assuming that $\{\mathbf{g}_{k,\nu_n}\}_{\substack{k=1,2,\ldots, \kappa \\ n=1,2,\ldots,\mathfrak{n}}}$ is a frame for
$\ell^2(\mathbf{G})$ we have that the rank of the matrix  $\mathbb{R}_{\mathsf{a},\mathbf{b}}$ is  $\mathfrak{g}$. Let $\mathbb{M}$ be a left-inverse of the matrix 
$\mathbb{R}_{\mathsf{a},\mathbf{b}}$ whose columns are denoted by $\mathbf{m}_{k,\nu_n}$, $k=1,2,\dots,\kappa$ and $n=1, 2,\dots, \mathfrak{n}$, as in matrix \eqref{matrixG}. All these matrices are expressed as (see Ref.\cite{penrose:55}) 
\begin{equation}
\label{todasLI}
    	\mathbb{M}=\mathbb{R}^\dag_{\mathsf{a},\mathbf{b}}+\mathbb{U}[\mathbb{I}_{\mathfrak{n}\kappa}-\mathbb{R}_{\mathsf{a},\mathbf{b}}\mathbb{R}^\dag_{\mathsf{a},\mathbf{b}}]
\end{equation}
where the $\mathfrak{g}\times \mathfrak{n}\kappa$ matrix  
$\mathbb{R}^\dag_{\mathsf{a},\mathbf{b}}=[\mathbb{R}^*_{\mathsf{a},\mathbf{b}}\mathbb{R}_{\mathsf{a},\mathbf{b}}]^{-1}\mathbb{R}^*_{\mathsf{a},\mathbf{b}}$ is  the Moore-Penrose  pseudoinverse of $\mathbb{R}_{\mathsf{a},\mathbf{b}}$ 
(see \cite{penrose:55}), and $\mathbb{U}$ denotes an arbitrary $\mathfrak{g}\times \mathfrak{n}\kappa$ matrix.
From \eqref{samplematrix} we obtain the frame expansion
\begin{equation}
\label{samplingFormulaGen}
\boldsymbol{\alpha}=\mathbb{M}\,\Lc_{\operatorname{samp}}\mathsf{f}=
\sum_{k=1}^\kappa\sum_{n=1}^\mathfrak{n}\Lc_k \mathsf{f}(\nu_n)\mathbf{m}_{k,\nu_n}=
\sum_{k=1}^\kappa \sum_{n=1}^{\mathfrak{n}}\langle \boldsymbol{\alpha},\mathbf{g}_{k,\nu_n}\rangle_{\ell^2(\mathbf{G})}\,
\mathbf{m}_{k,\nu_n}\,.
\end{equation}

\subsection*{The sampling result}
\label{subsection:SamplingResult}
Let $\mathsf{f}=\sum_{g\in \mathbf{G}} \alpha_g \,U(g)\mathsf{a}$ be a vector of $\Ac_\mathsf{a}$;  applying the isomorphism \eqref{iso} in \eqref{samplingFormulaGen} we get
\begin{equation}
\label{samplingResult}
   \mathsf{f}=\Tc^\mathbf{G}_\mathsf{a}(\boldsymbol{\alpha})=
   \sum_{k=1}^\kappa\sum_{n=1}^\mathfrak{n} \Lc_k \mathsf{f}(\nu_n)\,\Tc^\mathbf{G}_\mathsf{a}(\mathbf{m}_{k,\nu_n})
\end{equation}
The columns $\mathbf{m}_{k,\nu_n}$ in the formula above  do not have, in principle,  a suitable structure for applying the shifting property \eqref{shiftgroup}. Although we will see that the columns of the Moore-Penrose pseudo-inverse $\mathbb{R}^\dag_{\mathsf{a},\mathbf{b}}$ fulfil the required attribute, we  will construct in the next section all the left-inverses of $\mathbb{R}_{\mathsf{a},\mathbf{b}}$ allowing it.

\medskip

In order to prove that $\mathbb{R}^\dag_{\mathsf{a},\mathbf{b}}$ has the suitable structure, note that each $\mathfrak{n}\times \mathfrak{g}$ block $\mathbb{R}_{\mathsf{a},\mathsf{b}_k}$ has an $\mathfrak{h}$-circulant character in the sense that each row of $\mathbb{R}_{\mathsf{a},\mathsf{b}_k}$ is the previous row moved to the right $\mathfrak{h}$ places and wrapped around. In general, and in terms of  a matrix $C$  of order 
$\kappa \mathfrak{n}\times \mathfrak{g}$ partitioned into $\kappa$ submatrices of order $\mathfrak{n}\times \mathfrak{g}$, each block has a $\mathfrak{h}$-circulant character if and only if 
$ C P_{\mathfrak{g}}^{\mathfrak{h}}=   \mathbb{P} C$,  or equivalently, 
\[
 C= \mathbb{P}^* C P_{\mathfrak{g}}^{\mathfrak{h}}
 \]
where $P_i$ denotes the $1$-circulant square matrix of order $i\in \mathbb{N}$ with first row $(0, 1, 0, \cdots, 0)$ and $ \mathbb{P}$ is the square matrix of order 
$\kappa \mathfrak{n}$ given by  $\mathbb{P}= \text{diag}(P_\mathfrak{n}, \dots, P_\mathfrak{n})$,  the direct sum of $\kappa$ times the matrix $P_\mathfrak{n}$. The above characterization allows to conclude easily that $(C^\dag)^*$ inherits, and consequently $(C^\dag)^\top$, the $\mathfrak{h}$-circulant character from  $C$. Indeed 
\[
(C^\dag)^*=  ((\mathbb{P}^* C P_{\mathfrak{g}}^{\mathfrak{h}})^\dag)^* =\big((P_{\mathfrak{g}}^{\mathfrak{h}})^* C^{\dag}\mathbb{P}\big)^*=\mathbb{P}^* (C^\dag)^* P_{\mathfrak{g}}^{\mathfrak{h}}\,.
 \]
For more details on pseudoinverses of circulant matrices see Refs.~\cite{pye:73,stallings:72}. In these sources are to be found the  above results although  for a square matrix $C$.
%%%%%%%%%%%%%%%%%%%%%%%%%%%%%%%%%%
\subsection*{$\mathbf{G}$-compatible left-inverses} 
%%%%%%%%%%%%%%%%%%%%%%%%%%%%%%%%%%
Now, we proceed to construct a specific left-inverse $\mathbb{M}_{\mathbb{S}}$ of $\mathbb{R}_{\mathsf{a},\mathbf{b}}$ from any left-inverse $\mathbb{M}$ given by \eqref{todasLI} in the following way:
We denote by $\mathbb{S}$ the first $\mathfrak{h}$ rows of $\mathbb{M}$; i.e,
\begin{equation}
\label{SPartialEquation}
\mathbb{S}\mathbb{R}_{\mathsf{a},\mathbf{b}}=\big(\mathbb{I}_{\mathfrak{h}}\,\, \mathbb{O}_{\mathfrak{h}\times (\mathfrak{g}- \mathfrak{h})}\big)\,.
\end{equation}
Having in mind the structure of $\mathbb{R}_{\mathsf{a},\mathbf{b}}$ we write the $\mathfrak{h}\times \mathfrak{n}\kappa$ matrix $\mathbb{S}$ as
\[
\mathbb{S}=\big(\mathbb{S}_1 \, \mathbb{S}_2 \cdots \mathbb{S}_\kappa\big)
\]
where each block $\mathbb{S}_k$ is a $\mathfrak{h}\times \mathfrak{n}$ matrix denoted by $\mathbb{S}_k=\big(S_k^1\, S_k^2\, \dots\, S_k^{\mathfrak{n}}\big)$ where 
$S_k^n\in\CC^{\mathfrak{h}}$ for each  $n=1,2, \dots,\mathfrak{n}$ and $k=1,2,\dots,\kappa$. From \eqref{estructuraRCasoGeneral} and \eqref{SPartialEquation} we have
\begin{align*}
    &\sum_{k=1}^\kappa \sum_{n=1}^{\mathfrak{n}} S_k^n \,r_{\mathsf{a},\mathsf{b}_k}(\nu_1^{-1}\nu_n^{-1}\mathbf{H})=\mathbb{I}_{\mathfrak{h}}\\
    &\sum_{k=1}^\kappa \sum_{n=1}^{\mathfrak{n}} S_k^n \,r_{\mathsf{a},\mathsf{b}_k}(\nu_l^{-1}\nu_n^{-1}\mathbf{H})=\mathbb{O}_{\mathfrak{h}},\quad l=2,\dots,\mathfrak{n}\,,
\end{align*}
or equivalently
\begin{align}
\label{partialEq2Gen}
    &\sum_{k=1}^\kappa \sum_{n=1}^{\mathfrak{n}} S_k^n \,r_{\mathsf{a},\mathsf{b}_k}(\nu_n^{-1}\mathbf{H})=\mathbb{I}_{\mathfrak{h}}\\
\label{partialEq3Gen}
    &\sum_{k=1}^\kappa \sum_{n=1}^{\mathfrak{n}} S_k^n \,r_{\mathsf{a},\mathsf{b}_k}(\nu_n^{-1}\nu_l^{-1}\mathbf{H})=\mathbb{O}_{\mathfrak{h}},\quad l=2,\dots,\mathfrak{n}\,.
\end{align}
Now, we form the $\mathfrak{g}\times \mathfrak{n}\kappa$ matrix 
$\mathbb{M}_{\mathbb{S}}=\big(\widetilde{\mathbb{S}}_1\,\widetilde{\mathbb{S}}_2 \dots\widetilde{\mathbb{S}}_\kappa\big)$; each $\mathfrak{g}\times \mathfrak{n}$ block $\widetilde{\mathbb{S}}_k$, $k=1,2,\dots,\kappa$, is formed from the columns of $\mathbb{S}_k$ in the following manner:
\[
   \widetilde{\mathbb{S}}_k:=\begin{pmatrix}
                          S_k^1 & S_k^2 & \cdots & S_k^\mathfrak{n}\\
                          S_k^{1,2} & S_k^{2,2} & \cdots & S_k^{\mathfrak{n},2}\\
                          \vdots & \vdots &\cdots &\vdots\\
                          S_k^{1,\mathfrak{n}} & S_k^{2,\mathfrak{n}} & \cdots & S_k^{\mathfrak{n},\mathfrak{n}}
                        \end{pmatrix}
\]
where, for $i=2,\dots, \mathfrak{n}$ and $n=1, 2,\dots,\mathfrak{n}$, we set
\begin{equation}
\label{crucial}
S_k^{n,i}:=S_k^m\quad  \text{where \,$\nu_m\in \mathbf{N}$\, is the unique element such that  $\nu_m^{-1}=\nu_n^{-1}\nu_i^{-1}$} 
\end{equation}
\begin{lema}
The above $\mathfrak{g}\times \mathfrak{n}\kappa$ matrix $\mathbb{M}_{\mathbb{S}}$ is a left-inverse of $\mathbb{R}_{\mathsf{a},\mathbf{b}}$, i.e, $\mathbb{M}_{\mathbb{S}}\,\mathbb{R}_{\mathsf{a},\mathbf{b}}=\mathbb{I}_\mathfrak{g}$.
\end{lema}
\begin{proof}
From \eqref{partialEq2Gen},  for each $i=2,3,\dots,\mathfrak{n}$, we have
\[
\sum_{k=1}^\kappa \sum_{n=1}^{\mathfrak{n}} S_k^{n,i}\,r_{a,\mathsf{b}_k}(\nu_n^{-1}\nu_i^{-1}\mathbf{H})=
\sum_{k=1}^\kappa \sum_{m=1}^{\mathfrak{n}} S_k^m\,r_{\mathsf{a},\mathsf{b}_k}(\nu_m^{-1}\mathbf{H})= \mathbb{I}_\mathfrak{h}
\]
and, from \eqref{partialEq3Gen}
\[
\sum_{k=1}^\kappa \sum_{n=1}^{\mathfrak{n}} S_k^{n,i}\,r_{\mathsf{a},\mathsf{b}_k}(\nu_n^{-1}\nu_l^{-1}\mathbf{H})=
\sum_{k=1}^\kappa \sum_{m=1}^{\mathfrak{n}} S_k^m\,r_{\mathsf{a},\mathsf{b}_k}\big(\nu_m^{-1}(\nu_i \nu_l^{-1})\mathbf{H}\big)= \mathbb{O}_\mathfrak{h},\quad l\neq i\,,
\]
since $\nu_i \nu_l^{-1}\neq 1_\mathbf{G}$ for  $l\neq i$. As a consequence, we deduce that $\mathbb{M}_{\mathbb{S}}\,\mathbb{R}_{\mathsf{a},\mathbf{b}}=\mathbb{I}_\mathfrak{g}$.
\end{proof}

Denoting the columns of the matrix $\mathbb{M}_{\mathbf{S}}$ as $\widetilde{\mathbf{s}}_{k,n}$,  $k=1,2,\dots,\kappa$ and $n=1,2, \dots,\mathfrak{n}$\,, we have
\begin{equation}
\label{Stilde}
\mathbb{M}_{\mathbf{S}}=\Big(
\widetilde{\mathbf{s}}_{1,1}  \dots \widetilde{\mathbf{s}}_{1,\mathfrak{n}} \,\, \widetilde{\mathbf{s}}_{2,1} \dots \widetilde{\mathbf{s}}_{2,\mathfrak{n}}  \dots  \widetilde{\mathbf{s}}_{\kappa,1}  \dots \widetilde{\mathbf{s}}_{\kappa,\mathfrak{n}}\Big)
\end{equation}
Using the left-inverse $\mathbb{M}_{\mathbb{S}}$ of $\mathbb{R}_{\mathsf{a},\mathbf{b}}$ instead of $\mathbb{M}$ in \eqref{samplingFormulaGen}, for each 
$\mathsf{f}=\sum_{g\in \mathbf{G}} \alpha_g \,U(g)\mathsf{a}$ in $\Ac_\mathsf{a}$ we obtain
\[
\boldsymbol{\alpha}=\mathbb{M}_{\mathbb{S}}\,\Lc_{\operatorname{samp}}=
\sum_{k=1}^\kappa \sum_{n=1}^{\mathfrak{n}}\Lc_k\mathsf{f}(\nu_n)\,\widetilde{\mathbf{s}}_{k,n}\,.
\]
On the other hand, the columns $\widetilde{\mathbf{s}}_{k,n}$, $k=1,2,\dots,\kappa$ and 
$n=1,2,\dots,\mathfrak{n}$, as vectors of $\ell^2(\mathbf{G})$ satisfy, by construction, see \eqref{crucial}, the crucial property
\[
      \widetilde{\mathbf{s}}_{k,n}= L_{\nu_n} \widetilde{\mathbf{s}}_{k,1},\quad k=1,2,\dots,\kappa\,\text{ and }\, n=1, 2,\dots,\mathfrak{n}\,.
\]
Hence, the shifting property \eqref{shiftgroup} gives
\[
\begin{split}
\mathsf{f}&=\sum_{k=1}^\kappa\sum_{n=1}^{\mathfrak{n}}\Lc_k\mathsf{f}(\nu_n)\,\Tc^\mathbf{G}_{\mathsf{a}}(\widetilde{\mathbf{s}}_{k,n})
  =\sum_{k=1}^\kappa\sum_{n=1}^{\mathfrak{n}}\Lc_k\mathsf{f}(\nu_n)\,\Tc^\mathbf{G}_{\mathsf{a}}(L_{\nu_n}
                                                   \widetilde{\mathbf{s}}_{k,1})\\
  &=\sum_{k=1}^\kappa\sum_{n=1}^{\mathfrak{n}}\Lc_k\mathsf{f}(\nu_n)\,U(\nu_n)\, \widetilde{\mathbf{s}}_{k,1}\,.
\end{split}
\]
Therefore, we have proved that, for any $\mathsf{f}\in \Ac_\mathsf{a}$ the sampling expansion
\[
   \mathsf{f}=\sum_{k=1}^\kappa\sum_{n=1}^{\mathfrak{n}}\Lc_k\mathsf{f}(\nu_n)\,U(\nu_n)\,\mathsf{c}_k
\]
holds, where $\mathsf{c}_k=\Tc^\mathbf{G}_{\mathsf{a}}(\widetilde{\mathbf{s}}_{k,1})\in \Ac_\mathsf{a}$, $k=1,2,\dots,\kappa$. In fact, collecting all the pieces we have obtained until now we can state the following result: 
%%%%%%%%%%%%%%%%%%%%
\begin{teo}
\label{teoGeneral}
Consider the $\mathfrak{n}\kappa \times \mathfrak{g}$ matrix  $\mathbb{R}_{\mathsf{a},\mathbf{b}}$ defined in 
\eqref{matrixcc}. The following statements are equivalent:
\begin{enumerate}
\item $\rank \mathbb{R}_{\mathsf{a},\mathbf{b}}=\mathfrak{g}$
\item There exists a $\mathfrak{h} \times \mathfrak{n}\kappa$ matrix $\mathbb{S}$ such that 
    \[
      \mathbb{S}\,\mathbb{R}_{\mathsf{a},\mathbf{b}}=\big(\mathbb{I}_{\mathfrak{h}}\, \mathbb{O}_{\mathfrak{h}\times (\mathfrak{g}-\mathfrak{h})}\big)
    \]
\item There exist  vectors $\mathsf{c}_k\in\Ac_\mathsf{a}$, $k=1,2,\dots,\kappa$, such that  
$\big\{U(\nu_n)\mathsf{c}_k\big\}_{\substack{k=1,2,\ldots, \kappa \\ n=1, 2, \ldots,\mathfrak{n}}}$ is a frame (spanning set) for  $\Ac_\mathsf{a}$, and for any 
$\mathsf{f}\in\Ac_\mathsf{a}$ we have the expansion 
\[
\mathsf{f}=\sum_{k=1}^\kappa \sum_{n=1}^{\mathfrak{n}}\Lc_k\mathsf{f}(\nu_n)\,U(\nu_n)\mathsf{c}_k\,.
\]
\item There exists a frame $\{\mathsf{C}_{k,n}\}_{\substack{k=1,2,\ldots, \kappa \\ n=1,2,\ldots,\mathfrak{n}}}$ for $\Ac_\mathsf{a}$ such that, for each 
$\mathsf{f}\in \Ac_\mathsf{a}$ we have the expansion
\[
\mathsf{f}=\sum_{k=1}^\kappa \sum_{n=1}^{\mathfrak{n}}\Lc_k\mathsf{f}(\nu_n)\, \mathsf{C}_{k,n}\,.
\]
\end{enumerate}
\end{teo}
\begin{proof}
That condition $(1)$ implies condition $(2)$ and condition $(2)$ implies condition $(3)$ have been proved above. Obviously, condition $(3)$ implies condition $(4)$: take $\mathsf{C}_{k,n}=U(\nu_n)\mathsf{c}_k$ for 
$k=1, 2, \dots, \kappa$ and $n=1, 2, \dots, \mathfrak{n}$. Finally, as a consequence of Corollary \ref{cor1}, condition $(4)$ implies condition $(1)$.
\end{proof} 

For the particular case where $\kappa=\mathfrak{h}$ we obtain:
\begin{cor}
Assume that $\kappa=\mathfrak{h}$ and consider the $\mathfrak{g}\times \mathfrak{g}$ cross-covariance matrix $\mathbb{R}_{\mathsf{a},\mathbf{b}}$ defined in \eqref{matrixcc}. The following statements are equivalent:
\begin{enumerate}[(i)]
\item The  matrix $\mathbb{R}_{\mathsf{a},\mathbf{b}}$ is invertible.
\item There exist $\mathfrak{h}$ unique elements $\mathsf{c}_k\in \Ac_\mathsf{a}$, $k=1, 2, \dots, \mathfrak{h}$, such that the sequence $\big\{U(\nu_n) \mathsf{c}_k\big\}_{\substack{k=1,2,\ldots, \mathfrak{h} \\ n=1,2,\ldots,\mathfrak{n}}}$ is a basis for $\Ac_\mathsf{a}$, and  the expansion of any $\mathsf{f}\in \Ac_\mathsf{a}$ with respect to this basis is
\[
\mathsf{f}=\sum_{k=1}^{\mathfrak{h}} \sum_{n=1}^{\mathfrak{n}} \Lc_k \mathsf{f}(\nu_n)\, U(\nu_n) \mathsf{c}_k\,.
\]
\end{enumerate}
In case the equivalent conditions are satisfied, the interpolation property $\Lc_{k}\mathsf{c}_{k'}(\nu_n)=\delta_{k,k'}\,\delta_{n,1}$, \, $n=1,2, \dots ,\mathfrak{n}$ and  $k,k'=1,2, \dots, \mathfrak{h}$, holds.
\end{cor}
\begin{proof}
Notice that the inverse matrix $\mathbb{R}_{\mathsf{a},\mathbf{b}}^{-1}$ has necessarily the structure of the matrix $\mathbb{M}_{\mathbb{S}}$ in \eqref{Stilde}. The uniqueness of the expansion with respect to a basis gives the interpolation property.
\end{proof}

%%%%%%%%%%%%%%%%%%%%%%%%%%%%%%%%%%%%%%%%%
\subsection{Case of $\mathbf{H}$ Abelian subgroup: sampling indexed by $\mathbf{H}$}
%%%%%%%%%%%%%%%%%%%%%%%%%%%%%%%%%%%%%%%%%
In case $\mathbf{G}=\mathbf{N}\bowtie \mathbf{H}$, there is just one element of $\mathbf{H}$ in each left coset of $\mathbf{G}/\mathbf{N}$. Indeed, suppose that there exists $\tau,\,\tau^*\in \mathbf{H}$ such that $\tau=\tau^*\nu$ for some $\nu \in \mathbf{N}$. Then, 
 \[
     \tau=\tau^* \nu=\alpha_{\tau^*}(\nu)\beta_{\nu}(\tau^*),
 \]
and,  as a consequence,  $\alpha_{\tau^*}(\nu)=1_{\mathbf{G}}$, which implies $\nu=1_{\mathbf{G}}$, since
 $\alpha_{\tau^*}\in Aut(\mathbf{N})$, and $\tau=\tau^*$. Hence, we can choose an element of $\mathbf{H}$ in each left coset of the quotient set $\mathbf{G}/\mathbf{N}$. As $\mathfrak{h}=|\mathbf{H}|=|\mathbf{G}/\mathbf{N}|$, in case $\mathbf{H}=\{\tau_1=1_{\mathbf{G}},\, \tau_2,\dots,\tau_\mathfrak{h}\}$ we can describe the quotient set $\mathbf{G}/\mathbf{N}$ as
\begin{equation}
\label{G/N}
\mathbf{G}/\mathbf{N}=\big\{[1_{\mathbf{G}}=\tau_1], [\tau_2], \cdots, [\tau_\mathfrak{h}] \big\}\,.
\end{equation}
Hence, for a fixed order of the elements in $\mathbf{N}$ we  can write $\mathbf{G}=\bigcup_{n=1}^\mathfrak{h} \tau_n^{-1}\mathbf{N}$.  
Therefore, in case the subgroup $\mathbf{H}$ is Abelian the above partition of  $\mathbf{G}$ allow us to proceed as in the previous section in order to obtain a sampling formula for $\Ac_\mathsf{a}$ by indexing the data samples in $\mathbf{H}$. 

Indeed, for fixed $\kappa$ elements $\mathsf{b}_k\in\Hc$, $k=1,2,\dots,\kappa$, for each $\mathsf{f}=\sum_{g\in \mathbf{G}}\alpha_g \,U(g)\mathsf{a}$ in 
$\Ac_\mathsf{a}$ we define again its generalized samples, now indexed in $\mathbf{H}$, by
\begin{equation}
\label{gsamplesH}
\Lc_k \mathsf{f} (\tau_n):=\big\langle \mathsf{f},U(\tau_n)\mathsf{b}_k\big \rangle_\Hc\,, \quad \text{$n=1, 2, \dots, \mathfrak{h}$ and $k=1,2, \dots,\kappa$}\,.
\end{equation}
Roughly speaking, to recover any $\mathsf{f} \in \Ac_\mathsf{a}$ we need at least $\mathfrak{g}$ samples; if we are sampling at $\mathbf{H}$, we will need at least  
$\kappa \mathfrak{h} \geq \mathfrak{g}=\mathfrak{n} \mathfrak{h}$ samples, i.e., $\kappa \geq \mathfrak{n}$.

Let $\mathsf{f}$ be in $\Ac_\mathsf{a}$, in this case the expression for $\Lc_k\mathsf{f}(\tau_n)$ with $n=1, 2, \dots, \mathfrak{h}$, similar to \eqref{genSamples}, is
\begin{equation}
\label{genSamplesH}
\Lc_k\mathsf{f}(\tau_n)=\big\langle \boldsymbol{\alpha},\mathbf{g}_{k,\tau_n}\big\rangle_{\ell^2(\mathbf{G})}\,,
\end{equation}
where $\boldsymbol{\alpha}=(\alpha_g)_{g\in \mathbf{G}}$ and 
$\mathbf{g}_{k,\tau_n}=\big(\,\overline{\langle U(g)\mathsf{a},U(\tau_n)\mathsf{b}_k\rangle }\,\big)_{g\in \mathbf{G}}$ 
belong to $\ell^2(\mathbf{G})$. The vectors $\mathbf{g}_{k,\tau_n}\in \ell^2(\mathbf{G})$, 
$k=1,2,\dots,\kappa$, $n=1, 2, \dots, \mathfrak{h}$, can be expressed in terms of the cross-covariance $r_{\mathsf{a},\mathsf{b}_k}$ as 
$\mathbf{g}_{k,\tau_n}=\big(\,\overline{r_{\mathsf{a},\mathsf{b}_k}}(\tau_n^{-1}g)\,\big)_{g\in \mathbf{G}}$.
Now,  proceeding as before, since $\mathbf{H}$ is abelian, for each $k=1,2,\dots,\kappa$, we get  the $\mathfrak{h}\times \mathfrak{g}$ matrix $\mathbb{R}_{\mathsf{a},\mathsf{b}_k}$ 
\begin{equation}
 \label{estructuraRCasoGeneralH}
  \mathbb{R}_{\mathsf{a},\mathsf{b}_k}=
 \begin{pmatrix}
          r_{\mathsf{a},\mathsf{b}_k}(\tau^{-1}_{1}\tau_1^{-1}\mathbf{N})&
              r_{\mathsf{a},\mathsf{b}_k}(\tau^{-1}_{2}\tau_1^{-1}\mathbf{N})&\dots& 
                              r_{\mathsf{a},\mathsf{b}_k}(\tau^{-1}_{\mathfrak{h}}\tau_1^{-1}\mathbf{N})\\
        r_{\mathsf{a},\mathsf{b}_k}(\tau^{-1}_{1}\tau_2^{-1}\mathbf{N})&
              r_{\mathsf{a},\mathsf{b}_k}(\tau^{-1}_{2}\tau_2^{-1}\mathbf{N})&\dots& 
                                r_{\mathsf{a},\mathsf{b}_k}(\tau^{-1}_{\mathfrak{h}}\tau_2^{-1}\mathbf{N})\\
          \vdots & \vdots & \cdots& \vdots\\
         r_{\mathsf{a},\mathsf{b}_k}(\tau^{-1}_{1}\tau_\mathfrak{h}^{-1}\mathbf{N})&
              r_{\mathsf{a},\mathsf{b}_k}(\tau^{-1}_{2}\tau_\mathfrak{h}^{-1}\mathbf{N})&\dots& 
                                 r_{\mathsf{a},\mathsf{b}_k}(\tau^{-1}_{\mathfrak{h}}\tau_\mathfrak{h}^{-1}\mathbf{N}) 
         \end{pmatrix}
\end{equation}
where $\tau_n\in \mathbf{H}$ and $\overline{r_{\mathsf{a},\mathsf{b}_k}}(\tau^{-1}_{n}\tau_m^{-1}\mathbf{N})$ is the row vector
\[
  \overline{r_{\mathsf{a},\mathsf{b}_k}}(\tau^{-1}_{n}\tau^{-1}_m \mathbf{N})=
                  \Big(\overline{r_{\mathsf{a},\mathsf{b}_k}}(\tau^{-1}_{n}\tau^{-1}_m\nu_1),
                        \overline{r_{\mathsf{a},\mathsf{b}_k}}(\tau^{-1}_{n}\tau^{-1}_m\nu_2),\dots,
                          \overline{r_{\mathsf{a},\mathsf{b}_k}}(\tau^{-1}_{n}\tau^{-1}_m\nu_{\mathfrak{n}})\Big)\,,
\]
where $\mathbf{N}=\{\nu_1=1_{\mathbf{G}},\, \nu_2,\dots,\nu_\mathfrak{n}\}$.
From matrices $ \mathbb{R}_{\mathsf{a},\mathsf{b}_k}$ in \eqref{estructuraRCasoGeneralH}, $k=1,2,\dots,\kappa$,  we form now the new  $\mathfrak{h}\kappa \times \mathfrak{g}$ cross-covariance matrix $\mathbb{R}_{\mathsf{a},\mathbf{b}}$ as in \eqref{matrixcc}:
\begin{equation}
\label{matrixccH}
\mathbb{R}_{\mathsf{a},\mathbf{b}}:= \begin{pmatrix}       
\mathbb{R}_{\mathsf{a},\mathsf{b}_1} \\
\mathbb{R}_{\mathsf{a},\mathsf{b}_2}\\  
\vdots \\
\mathbb{R}_{\mathsf{a},\mathsf{b}_\kappa}   \\
\end{pmatrix}
\end{equation}
From here, the sampling theory goes in the same manner as in section above. In fact we have the following result, completely analogous to Theorem \ref{teoGeneral}:
\begin{teo}
\label{teoGeneralH}
Consider the $\mathfrak{h}\kappa \times \mathfrak{g}$ matrix $\mathbb{R}_{\mathsf{a},\mathbf{b}}$ defined in \eqref{matrixccH}. The following statements are equivalent:
\begin{enumerate}
\item $\rank \mathbb{R}_{\mathsf{a},\mathbf{b}}=\mathfrak{g}$
\item There exists a $\mathfrak{n} \times \mathfrak{h}\kappa $ matrix $\mathbb{S}$ such that 
    \[
      \mathbb{S}\,\mathbb{R}_{\mathsf{a},\mathbf{b}}=\big(\mathbb{I}_{\mathfrak{n}}\, \mathbb{O}_{\mathfrak{n}\times (\mathfrak{g}-\mathfrak{n})}\big)
    \]
\item There exist  vectors $\mathsf{d}_k\in\Ac_\mathsf{a}$, $k=1,2,\dots,\kappa$, such that  
$\big\{U(\tau_n)\mathsf{d}_k\big\}_{\substack{k=1,2,\ldots, \kappa \\ n=1,2,\ldots,\mathfrak{h}}}$ is a frame for  $\Ac_\mathsf{a}$, and for any 
$\mathsf{f}\in\Ac_\mathsf{a}$ we have the expansion 
\[
\mathsf{f}=\sum_{k=1}^\kappa\sum_{n=1}^{\mathfrak{h}}\Lc_k\mathsf{f}(\tau_n)\,U(\tau_n)\mathsf{d}_k\,.
\]
\item There exists a frame $\{\mathsf{D}_{k,n}\}_{\substack{k=1,2,\ldots, \kappa \\ n=1,2,\ldots,\mathfrak{h}}}$ for $\Ac_\mathsf{a}$ such that, for each $\mathsf{f}\in \Ac_\mathsf{a}$ we have the expansion
\[
\mathsf{f}=\sum_{k=1}^\kappa\sum_{n=1}^{\mathfrak{h}}\Lc_k\mathsf{f}(\tau_n)\, \mathsf{D}_{k,n}\,.
\]
\end{enumerate}
\end{teo}
%%%%%%%%%%%%%%%%%%%%%%%%%%%%%%%%%%%%%%%%%%%%%
\section{An illustrative example: sampling associated with the dihedral group $D_{2N}$}
\label{section4}
%%%%%%%%%%%%%%%%%%%%%%%%%%%%%%%%%%%%%%%%%%%%%%
In Geometry, the dihedral group $D_{2N}$ refers to the symmetries of the $N$-gon; it has order $2N$. 
The group $D_{2N}=\mathbf{N}\mathbf{H}$ where $\mathbf{N}=\{e,r,r^2,\dots,r^{N-1}\}$ is the group generated by $r$, the rotation of $2\pi/N$ radians and $\mathbf{H}=\{e,s\}$ being $s$ the axial reflection.  In this case $\mathbf{N}$ is a normal subgroup of $D_{2N}$ and, as a consequence, $D_{2N}$ is  commonly referred as the internal semidirect product of $\mathbf{N}$ and $\mathbf{H}$. Since the subgroups $\mathbf{N}$ and $\mathbf{H}$ of $D_{2N}$ are both abelians, the sampling theorems in Theorems  \ref{teoGeneral} and \ref{teoGeneralH} apply. 

\medskip

Let $D_{2N} \ni g \mapsto U(g)\in \Uc(\Hc)$ be a unitary representation of $D_{2N}$ on a Hilbert space $\Hc$. Fixed $\mathsf{a}\in \Hc$, we consider $\Ac_\mathsf{a}$ the  subspace of $\Hc$ spanned by $\big\{U(g)\mathsf{a}\,:\, g\in D_{2N}\big\}$. In case this set is linearly independent in $\Hc$ it can be described as
\[
\Ac_\mathsf{a}=\Big\{\sum_{g\in D_{2N}} \alpha_g U(g)\mathsf{a} \,:\, \alpha_g\in \CC\Big\}\subset \Hc\,.
\]
Assume that $\kappa$ $\Lc_k$-systems are defined on $\Ac_\mathsf{a}$ from $\kappa$ vectors $\mathsf{b}_k \in \Hc$, $k=1, 2, \dots, \kappa$, as in \eqref{gsamplesH}. In this case, each $2\times 2N$ block $\mathbb{R}_{\mathsf{a},\mathsf{b}_k}$ of the $2\kappa \times 2N$ matrix $\mathbb{R}_{\mathsf{a},\mathbf{b}}$ in \eqref{matrixccH} is given by
\[
  \mathbb{R}_{\mathsf{a},\mathsf{b}_k}=
      \begin{pmatrix}
          r_{\mathsf{a},\mathsf{b}_k}(\mathbf{N})& r_{\mathsf{a},\mathsf{b}_k}(s \mathbf{N})\\
          r_{\mathsf{a},\mathsf{b}_k}(s \mathbf{N})& r_{\mathsf{a},\mathsf{b}_k}( \mathbf{N})
      \end{pmatrix}
\]
If $\rank \mathbb{R}_{\mathsf{a},\mathbf{b}}=2N$, according to Theorem \ref{teoGeneralH} there exist vectors $\mathsf{d}_k\in \Ac_\mathsf{a}$, $k=1, 2, \dots,\kappa$, with $\kappa \geq N$ such that the sequence $\big\{U(h) \mathsf{d}_k\big\}_{k=1,2,\ldots, \kappa;\, h\in \mathbf{H}}$ is a frame for $\Ac_\mathsf{a}$ and, for any 
$\mathsf{f}\in \Ac_\mathsf{a}$ we obtain the sampling expansion
\begin{equation*}
 \mathsf{f}=\sum_{k=1}^\kappa \sum_{h\in\mathbf{H}}\Lc_k\mathsf{f}(h)\,U(h)\mathsf{d}_k\,.
\end{equation*}
Moreover, $\mathsf{d}_k=\Tc_{\mathsf{a}}^{D_{2N}}(\widetilde{\mathbf{s}}_{k,e})$ where $\widetilde{\mathbf{s}}_{k,e}$, $k=1,2,\dots,\kappa$, denote the corresponding column of a $2N\times 2\kappa$ left-inverse $\mathbb{M}_{\mathbb{S}}$ of the matrix $\mathbb{R}_{\mathsf{a},\mathbf{b}}$.

Analogously, we can take generalized samples indexed on $\mathbf{N}$ as in \eqref{gsamples} from $\kappa$ vectors $\mathsf{b}_k \in \Hc$, $k=1,2,\dots,\kappa$. In this case each $2\times 2N$ block $\mathbb{R}_{\mathsf{a},\mathsf{b}_k}$ of the $2\kappa\times 2N$ matrix $\mathbb{R}_{\mathsf{a},\mathbf{b}}$ in \eqref{matrixcc} is given by
\[
  \mathbb{R}_{\mathsf{a},\mathsf{b}_k}=
      \begin{pmatrix}
          r_{\mathsf{a},\mathsf{b}_k}(\mathbf{H})& r_{\mathsf{a},\mathsf{b}_k}(r \mathbf{H}) & \cdots& r_{\mathsf{a},\mathsf{b}_k}(r^{n-1}\mathbf{H})\\
          r_{\mathsf{a},\mathsf{b}_k}(r^{n-1}\mathbf{H})& r_{\mathsf{a},\mathsf{b}_k}(r^{n-2} \mathbf{H}) & \cdots& r_{\mathsf{a},\mathsf{b}_k}(\mathbf{H})\\
          \vdots             &\vdots                &\cdots & \vdots\\
          r_{\mathsf{a},\mathsf{b}_k}(r \mathbf{H})& r_{\mathsf{a},\mathsf{b}_k}(\mathbf{H}) & \cdots& r_{\mathsf{a},\mathsf{b}_k}(r^2 \mathbf{H})
      \end{pmatrix}
\]
If $\rank \mathbb{R}_{\mathsf{a},\mathbf{b}}=2N$, according to Theorem \ref{teoGeneral} there exist vectors 
$\mathsf{c}_j\in \Ac_\mathsf{a}$, $k=1, 2, \dots, \kappa$, with $\kappa \geq 2$ such that the sequence $\big\{U(n) \mathsf{c}_k\big\}_{k=1,2,\ldots, \kappa;\,n\in \mathbf{N}}$ is a frame for 
$\Ac_\mathsf{a}$ and, for any $\mathsf{f}\in \Ac_\mathsf{a}$ we have the sampling expansion
\begin{equation*}
 \mathsf{f}=\sum_{k=1}^\kappa \sum_{n\in \mathbf{N}} \Lc_k\mathsf{f}(n)\,U(n)\mathsf{c}_k\,.
\end{equation*}
Moreover, $\mathsf{c}_k=\Tc_{\mathsf{a}}^{D_{2N}}(\widetilde{\mathbf{s}}_{k,e})$ where $\widetilde{\mathbf{s}}_{k,e}$, $k=1,2,\dots,\kappa$\,, denote the corresponding column of a $2N\times 2\kappa$ left-inverse $\mathbb{M}_{\mathbb{S}}$ of $\mathbb{R}_{\mathsf{a},\mathbf{b}}$.

\medskip

\noindent{\bf Acknowledgments:} 
This work has been supported by the grant MTM2017-84098-P from the Spanish {\em Ministerio de Econom\'{\i}a y Competitividad (MINECO)}.

\vspace*{0.3cm}
%%%%%%%%%%%%%%%%%%%%%%%%%%%%%%%%%%%%%%%%%%%%%%%%%%%%%%%%%%%%%%%%%%%%%

%%%%%%%%%%%%%%%%%%%%%%%%%
\end{document}